\documentclass[11pt]{amsart}
\usepackage{amssymb, amscd, latexsym, color, amscd, float}
\usepackage{url}
\usepackage[final]{graphicx}

\usepackage{ulem,ifthen}
\newcommand{\ifemptythenelse}[3]{%
  \begingroup
    \def\dummy{#1}%
    \def\empty{}%
    \ifx\dummy\empty{#2}\else{#3}\fi
  \endgroup
  }

\newtheorem{thm}{Theorem}[section]
\newtheorem{prop}[thm]{Proposition}
\newtheorem{lem}[thm]{Lemma}
\newtheorem{cor}[thm]{Corollary}

\newtheorem{defn}[thm]{Definition}
\newtheorem{rmk}[thm]{Remark}

\begin{document}

\vspace*{2cm}


\title[\large On subgroup conjugacy separability]{\large On subgroup conjugacy separability of hyperbolic $\mathcal{Q}\mathcal{V}\mathcal{H}$-groups}

\author{Oleg Bogopolski}
\address{{Sobolev Institute of Mathematics of Siberian Branch of Russian Academy
of Sciences, Novosibirsk, Russia}\newline {and D\"{u}sseldorf University, Germany}}
\email{Oleg$\_$Bogopolski@yahoo.com}

\author{Kai-Uwe Bux}
\address{Bielefeld University, Germany}
\email{bux$\_$2009@kubux.net}


\thanks{$\dagger$ 
This research was partially supported by
SFB 701, ``Spectral Structures and Topological Methods in Mathematics'', at Bielefeld University.}


\begin{abstract}
A group $G$ is called subgroup conjugacy separable (abbreviated as SCS) if
any two finitely generated and non-conjugate subgroups of
$G$ remain non-conjugate in some finite quotient of $G$. An into-conjugacy version of SCS
is abbreviated by SICS.
We prove that if $G$ is a hyperbolic group, $H_1$ is a quasiconvex subgroup of $G$, and $H_2$
is a subgroup of $G$ which is elementwise conjugate into $H_1$, then
there exists a finite index subgroup of $H_2$ which is conjugate into~$H_1$.
As corollary, we deduce that fundamental groups of closed hyperbolic 3-manifolds and  torsion-free small cancellation groups with finite $C'(1/6)$ or $C'(1/4)-T(4)$ presentations are hereditarily quasiconvex-SCS and hereditarily quasiconvex-SICS, and that surface groups are SCS and SICS. We also show that the word ``quasiconvex''
cannot be deleted for at least small cancellation groups.

\end{abstract}

\maketitle

\section{Introduction}

A group $G$ is called {\it subgroup conjugacy separable} (SCS) if the following
condition holds:

\medskip
For any two finitely generated subgroups $H_1$ and $H_2$ that are not conjugate in
$G$, there is a homomorphism of $G$ onto a finite quotient $\overline{G}$ such that the images of $H_1$ and $H_2$
are not conjugate in $\overline{G}$.

\medskip

This property logically continues the following series of well
known properties of groups: residual finiteness, conjugacy separability (CS), and subgroup separability (LERF).
Note that SCS-groups are residually finite, but there are residually finite, and even conjugacy separable groups,
which are not SCS-groups. The SCS-property is relatively new and not much is known about, which groups enjoy this property.

In~\cite{GS}, Grunewald and Segal proved that all virtually polycyclic groups are SCS (see also Theorem~7 in Chapter~4 of~\cite{Segal}). In the preprint~\cite{BG}, Bogopolski and Grune\-wald proved that free groups and some virtually free groups are SCS.
In the preprint~\cite{BogBux}, Bogopolski and Bux proved that
the fundamental groups of closed orientable surfaces are SCS.
Chagas and Zalesskii~\cite{CZ3} generalized this to limit groups.
To formulate our results, we need the following definitions:

\begin{enumerate}
\item[(1)] For two subgroups $A$ and $B$ of a group $C$, we say that $A$ is {\it conjugate into}~$B$ if
$A^c\leqslant B$ for some element $c\in C$. Here $A^c=c^{-1}Ac$.

A group $G$ is called {\it subgroup into conjugacy separable} (SICS) if the following
condition holds:

\noindent
For any two finitely generated subgroups $H_1$ and $H_2$ such that $H_2$ is not conjugate into $H_1$ in $G$, there is a homomorphism from $G$ onto a finite quotient $\overline{G}$ such that the image of $H_2$ is not conjugate into the image of $H_1$ in $\overline{G}$.

\medskip

\item[(2)] The class of groups possessing {\it local retractions} is defined in Definition~\ref{LR}.
This class was introduced by Long and Ride~\cite[Introduction]{LR} and is denoted by ${\rm LR}$.
Results of Hall~\cite{Hall} and Scott~\cite{Scott} imply that free groups and closed surface groups lie in the class ${\rm LR}$. Wilton~\cite[Theorem B]{W} proved that limit groups lie in the class ${\rm LR}$.

\medskip

\item[(3)] For hyperbolic groups the notion `a quasiconvex subgroup' is well defined, i.e., it does not depend on the choice of generating sets. It is well known that quasiconvex subgroups of hyperbolic groups are finitely generated. We say that a hyperbolic group $G$ is {\it quasiconvex}-SICS ({\it quasiconvex}-SCS) if it satisfies the above definitions of SICS (SCS) with replacing the words finitely generated by quasiconvex.

\medskip

\item[(4)] Let ${\mathcal P}$ be a property of a group. We say that $G$ is {\it hereditarily} ${\mathcal P}$ if all finite index subgroups of $G$ have the property ${\mathcal P}$.

\medskip

\end{enumerate}


Our main theorem is the following:

\medskip

{\bf Theorem A.} (see Theorem~\ref{hyperb})
{\it Let $G$ be a hyperbolic group, let $H_1$ be a quasiconvex subgroup of $G$, and let $H_2$
be an arbitrary subgroup of $G$. Suppose that $H_2$ is elementwise conjugate into $H_1$. Then
there exists a finite index subgroup of $H_2$ which is conjugate into~$H_1$.}

\medskip

We deduce from this two corollaries:

\medskip

{\bf Corollary B.} (see Corollary~\ref{Dezember1})
{\it Let $G$ be a torsion-free hyperbolic group with the property~LR.
If $G$ is hereditarily CS, then $G$ is hereditarily SICS and SCS.}

\medskip

Our argument does not allow to leave out the word heridetarily.

The following corollary is about hyperbolic groups in the class $\mathcal{Q}\mathcal{V}\mathcal{H}$, see Definition~\ref{QVH}.
This class was introduced by Wise~\cite[Definition 11.5]{Wise2} in the process of solution of the virtual Haken conjecture (see the surveys of Bestvina~\cite{Bestvina} and Friedl~\cite{Friedl_2}).

\medskip

{\bf Corollary C.} (see Corollary~\ref{main_cor})
{\it Every torsion-free hyperbolic groups from the class $\mathcal{Q}\mathcal{V}\mathcal{H}$
is hereditarily  quasiconvex-SICS and hereditarily  quasiconvex-SCS.\break In particular,
the following groups possess these properties:

\begin{enumerate}


\item[(1)] Fundamental groups of closed hyperbolic {\rm 3}-manifolds.

\item[(2)] Torsion-free groups with finite $C'(1/6)$ or $C'(1/4)-T(4)$ presentations.

\item[(3)] Surface groups.
\end{enumerate}
}

\medskip

{\bf Remark.}
{\rm
(1) An alternative proof of a variant of Corollary C is given by Chagas and Zalesskii in~\cite[Theorem~1.2]{CZ3}.

(2) As mentioned above, every quasiconvex subgroup of a hyperbolic group is finitely generated.
The converse is true for all surface groups, but is not true for some small cancellation groups; and it is not true for all fundamental groups of closed hyperbolic 3-manifolds (see Remark~\ref{U-Bahn}).
Thus, surface groups are hereditarily SICS and hereditarily SCS;
this was already proved in~\cite{BogBux} by another method.

Theorem~D~\!(2) below shows that not all
torsion-free groups with finite $C'(1/6)$ presentations are SICS/SCS.
We do not know whether fundamental groups of hyperbolic 3-manifolds are SICS/SCS.

}


\medskip

The following theorem shows that, in general, CS does not imply SICS and SCS, and that SCS does not imply SICS.

\medskip

{\bf Theorem~D.} (see Proposition~\ref{SCS, but not cyclic-SICS} and Corollary~\ref{CS but not SICS/SCS})

\begin{enumerate}
\item[{\rm (1)}] {\it There exists a finitely generated torsion free nilpotent group which is SCS, but not SICS.}

\medskip

\item[{\rm (2)}] {\it There exists a torsion-free group with finite $C'(1/6)$ presentation that
is hereditarily CS, but is not SICS and is not SCS.
}
\end{enumerate}

\medskip

The structure of the paper is the following. In Section 2 we give an information on quasi-convex subgroups and retracts. Since our strategy is to deduce SCS from SICS,
we study in Section~3, in which cases this is possible.
In particular, we show that SCS follows from SICS in the class of LR-groups,
and that qusiconvex-SCS follows from qusiconvex-SICS in the class of hyperbolic groups.

In most applications we work with groups from
the class LR. Therefore we study in Section~4, when SICS can be carried to a finite index overgroup.
In Section~5 we summarize definitions and results on $\mathcal{Q}\mathcal{V}\mathcal{H}$-groups.
In Section~6, we deduce Corollaries~B and~C from a corollary to Theorem~A.
Theorem~A is proven in Section~7.

In Section~8, we discuss algorithmic problems concerning SICS and SCS.
We deduce Theorem~D from undecidability of some of these problems and
with the help of constructions of Segal and Rips.










\section{Quasiconvex subgroups and retracts}

\begin{defn} {\rm

(1) Let $\epsilon\geqslant 0$. A subset $Y$ of a geodesic metric space $X$ is called {\it $\epsilon$-quasiconvex}
if every geodesic with ends in $X$ lies in the $\epsilon$-neighborhood of $X$.\break
A subset $Y\subseteq X$ is called {\it quasiconvex}, if $Y$ is $\epsilon$-quasi-convex for some $\epsilon\geqslant 0$.

(2) A subgroup $H$ of a hyperbolic group $G$ is called {\it quasiconvex} if $H$ is a quasi-convex subset
in the Cayley graph $\Gamma(G,Y)$ with respect to some (and then any) finite generating set $Y$ of $G$
(see Remark~\ref{quasi-convex_and quasi-isom}).

(3) Let $\lambda\geqslant 1$ and $\epsilon\geqslant 0$. Let $(Z_1,d_1)$ and $(Z_2,d_2)$ be two metric spaces. A map $f:Z_1\rightarrow Z_2$ is called a {\it  $(\lambda,\epsilon)$-quasi-isometric embedding} if for every $x,y\in Z_1$ holds
$$\frac{1}{\lambda}\cdot d_2(f(x),f(y))-\epsilon \leqslant d_1(x,y)\leqslant \lambda\cdot d_2(f(x),f(y))+\epsilon$$
A map $f:Z_1\rightarrow Z_2$ is called a {\it quasi-isometric embedding} if it is a {\it $(\lambda,\epsilon)$-quasi-isometric embedding}
for some $\lambda\geqslant 1$ and $\epsilon\geqslant 0$.
}
\end{defn}

\begin{rmk}\label{quasi-convex_and quasi-isom}
{\rm Groups can be considered as metric spaces with respect to a chosen generating set.
In the category of finitely generated groups, the notion `a quasiconvex subgroup' is not well defined (it may happen that a f.g. subgroup $H$ of a f.g. group $G$ is quasiconvex with respect to some finite generating set of $G$ and is not quasiconvex with respect to another).
However, the notion `a quasi-isometric map' in this category is well defined.

If we restrict to hyperbolic groups $G$, then both notions are well defined and closely related:
A subgroup $H$ of a hyperbolic group $G$ is quasiconvex if and only if $H$ is finitely generated and
the natural embedding $H\hookrightarrow G$ is a quasi-isometric embedding. Moreover, quasiconvex subgroups of hyperboloic groups are hyperbolic.
}
\end{rmk}

\begin{rmk}\label{U-Bahn}
{\rm
\begin{enumerate}

\item[(i)] Every finitely generated subgroup of a surface group is quasi\-convex
by~\cite[Corolary 3.8]{Gitik_1}.

\item[(ii)] By the Rips construction~\cite{Rips}, there exists a torsion-free group $G$ with a finite $C'(1/6)$-presentation
that contains a finitely generated but not finitely presented subgroup $H$. This $H$ is not quasiconvex
since every quasiconvex subgroup of a hyperbolic group must be finitely presented.

\item[(iii)] For every closed hyperbolic 3-manifold $M$, the fundamental group $\pi_1(M)$ contains a finitely
generated but not a quasiconvex subgroup. Indeed, the virtual fibering conjecture, proved in~\cite[Theorem 9.2]{Agol}, says that
$M$ admits a finite-sheeted cover $\widetilde{M}\rightarrow M$ such that $\widetilde{M}$ fibers over the circle, where the fiber is a closed compact surface $S$.
Hence, there exists a short exact sequence
$$1\rightarrow \pi_1(S)\rightarrow \pi_1(\widetilde{M})\rightarrow \mathbb{Z}\rightarrow 1.$$
In particular, $\pi_1(S)$ is a normal subgroup of infinite index in $\pi_1(\widetilde{M})$.
But every infinite quasiconvex subgroup in a hyperbolic group has finite index in its normalizer
(see~\cite[Theorem 2]{MT}).
Thus, $\pi_1(S)$ is not quasiconvex in $\pi_1(\widetilde{M})$, and hence in $\pi_1(M)$.
\end{enumerate}
}
\end{rmk}

\begin{defn}\label{LR}
{\rm  A subgroup $H$ of a group $G$ is called a {\it retract} of $G$ if there exists an epimorphism $f:G\rightarrow H$ with $f_{|H}={\rm id}$. The epimorphism $f$ is called a {\it retraction}. Equivalently, $H$ is retract of $G$ if $G=K\rtimes H$ for some subgroup $K$ of $G$. A subgroup $H$ of $G$ is called a {\it virtual retract} of $G$ if $H$ is a retract of a finite index subgroup of~$G$. According to Long and Ride~\cite[Introduction]{LR},
a group $G$ has {\it property LR} (it {\it admits local retractions}) if every finitely generated subgroup $H\leqslant G$ is a virtual retract of $G$. We call such groups {\it LR-groups}.}
\end{defn}

Retracts have useful properties and applications.

\begin{lem}\label{observ}
Retracts in finitely generated groups are quasi-isometrically embedded into the ambient group. Retracts in hyperbolic groups are quasiconvex.
\end{lem}

\medskip

{\it Proof.} Let $f:G\rightarrow H$ be a retraction of groups.
Suppose that $Y$ is a finite generating set of $G$ containing a finite generating set $X$ of $H$.
Let $d_Y$ and $d_X$ be the corresponding word metrics on $G$ and $H$. Set $M=\underset{y\in Y}{\max} \, d_X(f(y))$.
To show that the natural embedding $H\hookrightarrow G$ is a quasi-isometric embedding, it suffices to show that for every $h\in H$ holds
$$d_Y(h)\leqslant d_X(h)\leqslant M\cdot d_Y(h).$$

The first inequality holds since $X\subseteq Y$. The second one follows from the fact that if
$h=y_1\dots y_n$ with factors from $Y\cup Y^{-1}$ and $n=d_Y(h)$, then $h=f(y_1)\dots f(y_n)$.
The last statement of this lemma follows from the last paragraph of Remark~\ref{quasi-convex_and quasi-isom}.
\hfill $\Box$

\begin{lem}\label{retracts_and_roots}  Let $G$ be a group with the unique root property.
Then retracts of $G$ are closed under taking of roots. In particular,
retracts of torsion-free hyperbolic groups are closed under taking of roots.
\end{lem}

{\it Proof.} Let $f:G\rightarrow H$ be a retraction.
Suppose that $g\in G$ is such that $g^n\in H$ for some $n\in \mathbb{N}$.
Then $f(g)^n=f(g^n)=g^n$. Since $G$ has the unique root property, we have $g=f(g)\in H$.

The last statement follows from the fact that torsion-free hyperbolic groups have the unique root property.
The latter follows from the fact that
nontrivial elements of torsion-free hyperbilic groups have cyclic centralizers
(see~\cite[Chapter III.$\Gamma$, Corollary~3.10\,(2)]{Br}).
\hfill $\Box$




\section{When SICS implies SCS}

We say that an automorphism $\alpha:G\rightarrow G$ of the group $G$ {\it expands} (or {\it contracts}) a subugroup $H\leqslant G$ if $H<H^{\alpha}$ (or $H>H^{\alpha}$, respectively) is a strict inclusion. We call an automorphism $\alpha:G\rightarrow G$ {\it expanding} (or {\it contracting}) if there is a finitely generated subgroup $H\leqslant G$ that is expanded (or contracted) by $\alpha$.

\begin{lem}\label{observation}
Let $H_1$ and $H_2$ be two finitely generated
subgroups of a group $G$.
Assume that $G$ does not admit
expanding inner automorphisms. If $H_2$ conjugates into $H_1$ and $H_1$ conjugates into $H_2$,
then $H_1$ and $H_2$ are conjugate. More precisely, for any two elements $g,h\in G$ with
$H_2^g\leqslant H_1$ and $H_1^h\leqslant H_2$ one already has equality: $H_2^g=H_1$ and $H_1^h=H_2$.
\end{lem}

{\it Proof.} We have $H_1\leqslant H_2^{h^{-1}}\leqslant H_1^{g^{-1}h^{-1}}$.
Put $f := g^{-1}h^{-1}$ and consider the associated
inner automorphism. Since it is not expanding, the inclusion $H_1 \leqslant H_1^f$
is not strict. Hence
$H_1 = H_1^f$ that implies $H_1^h=H_2$ and $H_2^g=H_1$.\hfill $\Box$

\begin{prop}
Suppose that a group $G$ does not admit
expanding inner automorphisms and is SICS, then it is SCS.
\end{prop}

{\it Proof.} Let $H_1$ and $H_2$ be two non-conjugate finitely generated subgroups of $G$. By
Lemma~\ref{observation}, $H_1$ is not conjugate into $H_2$ or $H_2$ is not conjugate into $H_1$. Both cases
are symmetric and we assume that $H_2$ is not conjugate into $H_1$.
Since $G$ is SICS, there exists a homomorphism from $G$ onto a finite group $\overline{G}$
such that the image of $H_2$ is not conjugate into the image of $H_1$ in $\overline{G}$.
In particular, the image of $H_2$ is not conjugate to the image of $H_1$ in $\overline{G}$. Hence, $G$ is SCS.
\hfill $\Box$

\begin{lem}
LR-groups do not admit expanding inner automorphisms.
\end{lem}

{\it Proof.} Suppose the contrary, that is,
there exists an LR-group $G$, a finitely generated subgroup $H$ of~$G$, and an element $g\in G$
such that $H$ is strictly contained in $H^g$.
By~definition of an LR-group, there exists a subgroup $G_0$ of finite index in $G$ containing $H$ and a retraction
$f:G_0\rightarrow \hspace*{-3mm}\rightarrow H$.
Then $g^n\in G_0$ for some $n\in \mathbb{N}$.
Since $H^{g^{-1}}<H\leqslant G_0$, we have $H^{g^{-n}}\leqslant G_0$. Moreover, $H^{g^{-n}}$ is strictly contained in $H$.
For each $h\in H$, we have $$h^{g^{-n}}=f(h^{g^{-n}})=f(h)^{f(g^{-n})}=h^{f(g^{-n})}.$$

Thus, restricted to $H$, conjugation by $g^{-n}$ is an inner automorphism
of $H$, namely conjugation by $f(g^{-n})$. Hence $H^{g^{-n}}=H$ and conjugation
by $g$ cannot expand $H$.

\hfill $\Box$

\begin{cor}\label{reduction_free}
If an LR-group is SICS, then it is SCS.
\end{cor}

\begin{rmk} {\rm Suppose that $G$ is a hyperbolic group and $H$ is an infinite quasiconvex subgroup of $G$. Mihalik and Towle show in~\cite{MT}, more precisely in item~(2) of their main theorem (formulated only in the abstract), that no inner automorphism of $G$ can contract (nor expand) $H$. Since finite subgroups cannot be expanded either, we deduce the following corollary.}
\end{rmk}


\begin{cor}\label{quasi-convex-SICS}
If a hyperbolic group is quasiconvex-SICS, then it is quasiconvex-SCS.
\end{cor}

\section{Passing to finite index overgroups}

For a compact formulation of further results, we need the following terminology:
Let $H_1$ and $H_2$ be finitely generated subgroups of $G$. We say that $H_1$ is
{\it con-sepa\-rated from $H_2$ within $G$} if there is a finite index subgroup $D\leqslant G$ containing $H_1$
such that $H_2$ is not conjugate into $D$. We call $D$ a {\it witness} of con-separation. Note that being
con-separated is not a symmetric relation. We call $H_1$ {\it con-separated in} $G$ if $H_1$
is con-separated from any finitely generated subgroup $H_2\leqslant G$ that is not already conjugate into $H_1$.

\begin{lem}\label{propJuly} Let $H_1,H_2$ be subgroups of a group $G$. Then the following conditions are equivalent:

\begin{enumerate}
\item[{\rm (1)}] $H_1$ is con-separated from $H_2$ within $G$;

\item[{\rm (2)}] There exists a homomorphism from $G$ onto a finite group $\overline{G}$ such that
the image of $H_2$ is not conjugate into the image of $H_1$.
\end{enumerate}
\end{lem}

{\it Proof.} $(1)\Rightarrow (2)$: Let $D$ be the witness of con-separation for $H_1$ from $H_2$.
Then $D$ contains a finite index subgroup $N$ which is normal in $G$.
Obviously, the image of $H_2$ in $G/N$ is not conjugate into the image of $H_1$.

$(2)\Rightarrow (1)$: If $\varphi:G\rightarrow \overline{G}$ is the homomorphism from (2), then $D:=H_1\cdot {\text{\rm ker}}\varphi$ is the witness of con-separation for $H_1$ from $H_2$.
$\hfill$ $\Box$

\medskip


The following lemma enables to push the con-separability within
a finite index subgroup to the con-separability within the whole group.

\begin{lem}\label{overgroups}
Let $G_1$ be a finite index subgroup of $G$, and
let $H_1\leqslant G_1$ and $H_2\leqslant G$ be finitely generated. Let $g_1,\dots ,g_k$
be a set of representatives for the left cosets $G/G_1$.
If $H_1$ is con-separated from $H_2^{g_i}$ in $G_1$ for each $i$ such that $H_2^{g_i}$
is a subgroup of $G_1$, then $H_1$ is con-separated from $H_2$ in $G$.
In particular, if $H_1$ is con-separated in $G_1$ it is also con-separated in $G$.
\end{lem}

{\it Proof.} If $H_2^{g_i}$ is contained in $G_1$,
let $D_i\leqslant G_1$ be a witness that $H_1$ is con-separated from
$H_2^{g_i}$ within $G_1$. Otherwise, put $D_i := G_1$. Note that in either case, $H_2^{g_i}$
is not conjugate into $D_i$ by a conjugating element of $G_1$.

\medskip
\noindent We claim that $H_1$ is con-separated from $H_2$ within $G$ with witness
$D := D_1 \cap \dots \cap D_k$. For contradiction, assume $H_2^g\leqslant D$ for some
$g\in G$. We write $g = g_ih$
for some $h\in G_1$. Then $H_2^{g_ih}\leqslant D\leqslant D_i$ whence $H_2^{g_i}$
would be conjugate into $D_i$ by a
conjugating element of $G_1$. This is a contradiction.\hfill $\Box$

\medskip

\begin{defn}\label{defn_CD}
{\rm A subgroup $H$ of $G$ is called {\it conjugacy distinguished} in $G$ if for every element $g\in G$
that is not conjugate into $H$, there exists a homomorphism $f:G\rightarrow \overline{G}$ with finite $\overline{G}$
such that $f(g)$ is not conjugate into $f(H)$.}
\end{defn}

\begin{rmk}\label{rmk_CD} {\rm The following definition is equivalent to Definition~\ref{defn_CD}:

A subgroup $H$ of $G$ is called {\it conjugacy distinguished} in $G$ if for every element $g\in G$
that is not conjugate into $H$, there exists a subgroup $D$ of finite index in $G$ such that
$D$ contains $H$ and $g$ is not conjugate into $D$.}

\end{rmk}

The following lemma can be proved analogously if we replace the subgroup $H_2\leqslant G$ by an element $h_2\in G$
and the words ``con-separated'' by ``conjugacy distinguished''.

\begin{lem}\label{overgroups_elem}
Let $G_1$ be a finite index subgroup of $G$, and
let $H_1\leqslant G_1$. If $H_1$ is conjugacy distinguished in $G_1$, then it is conjugacy distinguished in $G$.
\end{lem}




\maketitle

\section{The class $\mathcal{Q}\mathcal{V}\mathcal{H}$ introduced by D.T.~Wise}

This section aims to help the reader to keep track in a variety of new definitions and results that appeared
in the course of the Wise and Agol proof of the virtually Haken conjecture. We recommend short surveys of Bestvina~\cite{Bestvina} and Friedl~\cite{Friedl_2}, and the long preprint of Wise~\cite{Wise2} for interested reader. A substantial role in this subject play {\it cube complexes},
see~\cite{Br} for instance.
In 1995, Sageev~\cite{Sageev} introduced {\it nonpositively curved cube complexes}.
In 2008, Haglund and Wise~\cite{HW2} introduced {\it special cube complexes}. The definition of a special cube complex is rather technical, therefore we skip it here and concentrate on group-theoretical aspects.

\begin{defn} {\rm A group $G$ is said to be {\it virtually (compact) special} if it contains a finite
index subgroup $H$ such that $H=\pi_1(X)$ for some (compact) special cube complex $X$.}
\end{defn}

\begin{defn}\label{QVH}~\cite[Definition 11.5]{Wise2}
{\rm Let $\mathcal{Q}\mathcal{V}\mathcal{H}$ (the letters abbreviate the words {\it quasiconvex}, {\it virtual}, and {\it hierarchy})  denote the smallest class of groups that satisfies the following axioms.

\begin{enumerate}
\item[(1)] $1\in \mathcal{Q}\mathcal{V}\mathcal{H}$.\vspace*{1mm}

\item[(2)] If $G=A\ast_B C $ and $A,C\in \mathcal{Q}\mathcal{V}\mathcal{H}$ and $B$ is finitely generated and embeds by a quasi-isometric embedding into $G$, then $G$ is $\mathcal{Q}\mathcal{V}\mathcal{H}$.\vspace*{1mm}

\item[(3)] If $G=A\ast_B$ and $A\in \mathcal{Q}\mathcal{V}\mathcal{H}$ and $B$ is finitely generated and embeds by a quasi-isometric embedding into $G$, then $G$ is in $\mathcal{Q}\mathcal{V}\mathcal{H}$.\vspace*{1mm}

\item[(4)] If $G$ contains a finite index subgroup $H$ with $H\in \mathcal{Q}\mathcal{V}\mathcal{H}$,
 then $G\in \mathcal{Q}\mathcal{V}\mathcal{H}$.
\end{enumerate}
}
\end{defn}

\begin{rmk}\label{subgr_hyp}
{\rm By~\cite[Section 6, Theorems 4 and 6]{KM},
if $G\in \mathcal{Q}\mathcal{V}\mathcal{H}$ is hyperbolic, then the groups
at each step of the hierarchic construction related to $G$ are quasiconvex in $G$, and hence hyperbolic.
Therefore the class of hyperbolic $\mathcal{Q}\mathcal{V}\mathcal{H}$-groups coincides with the class
of $\mathcal{Q}\mathcal{V}\mathcal{H}$-groups defined by Agol in~\cite[Definition~2.6]{Agol}.}
\end{rmk}

\begin{thm}\label{Wise_Agol} {\rm (see~\cite[Theorem~13.3]{Wise2},~\cite[Theorem A.42]{Agol})} A hyperbolic group is in $\mathcal{Q}\mathcal{V}\mathcal{H}$ if and only if it is virtually compact special.
\end{thm}
We remark that in~\cite[Theorem~13.3]{Wise2}, the word ``compact'' is left out. However, the proof given there, using~\cite[Theorem~13.1]{Wise2} and induction along the hierarchy, yields the result as stated above.


There is a geometric source for virtually compact special groups:

\begin{thm}\label{Agol_1} {\rm (see~\cite[Theorem 1.1]{Agol})}
Let $G$ be a hyperbolic group acting properly and cocompactly on a CAT(0) cube complex $X$.
Then $G$ has a finite index subgroup $F$ acting specially on $X$.
\end{thm}

We observe that such $G$ is virtually compact special provided that, additionally, $G$ is virtually torsion-free: the finite index subgroup $F$ from Theorem~\ref{Agol_1} acts freely and specially on $X$; thus the quotient $F\setminus X$ ist non-positively curved, and so it is special by \cite[Lemma~9.12]{HW2}.

We summarize some known facts on virtually compact special groups in the following lemma.

\begin{lem}\label{examples_1}
The following groups are virtually compact special:
\begin{enumerate}

\item[{\rm (1)}] Hyperbolic groups in the class $\mathcal{Q}\mathcal{V}\mathcal{H}$.\\
(See {\rm Theorem~\ref{Wise_Agol}}.)
\vspace*{1mm}

\item[{\rm (2)}] Fundamental groups of closed hyperbolic {\rm 3}-manifolds.\\
(This is a combination of results of several authors. See the survey of Friedl~\cite[Theorem 7.13, Figure~6, and around]{Friedl_2}; see the whole~\cite[Section 7]{Friedl_2} for other classes of 3-manifold groups.)
\vspace*{1mm}

\item[{\rm (3)}] Hyperbolic Coxeter groups.\\
(Niblo and Reeves showed in~\cite[Theorem~1]{NR} that any finitely generated Coxeter group~$G$
acts properly on a CAT(0) cube complex $X$.
Moreover, this action is cocompact if $G$ is a hyperbolic Coxeter group \cite[Theorem~5]{NR}.
Haglund and Wise proved in~\cite[Theorem~1.2]{HW1} that any finitely generated Coxeter group $G$ has a finite index torsion-free subgroup
$F$ such that $F\setminus X$ is special.)

\vspace*{1mm}

\item[{\rm (4)}] Virtually torsion-free groups with finite $C'(1/6)$ or $C'(1/4)-T(4)$  presentations.\\
 (Indeed, in~\cite[Theorem 1.2]{Wise1} Wise showed that such groups act properly and cocompactly on CAT(0) cube complexes.
 Then apply {\rm Theorem~\ref{Agol_1}} and the observation after it.)
\vspace*{1mm}


\item[{\rm (5)}] One-relator torsion groups.\\
(By Newman' spelling theorem
(see~\cite[Chapter~IV, Theorem~5.5]{LS}) such groups
have Dehn's presentation, and hence they are hyperbolic; they
lie in the class $\mathcal{Q}\mathcal{V}\mathcal{H}$ by~\cite[Corollary 18.2]{Wise2}).
\vspace*{1mm}

\item[{\rm (6)}] Surface groups.\\
(These groups evidently lie in the class $\mathcal{Q}\mathcal{V}\mathcal{H}$.)
\end{enumerate}
\end{lem}

Note, that if a hyperbolic group is residually finite, then it is virtually torsion free since it has only finitely many conjugacy classes of torsion elements. In this case
the words `virtually torsion free' in (4) can be omitted.

We will use the following theorems.

\begin{thm}\label{herid_CS}\cite[Theorem 1.1]{MZ1} Let $G$ be a hyperbolic group in the class $\mathcal{Q}\mathcal{V}\mathcal{H}$.
Then $G$ is hereditarily conjugacy separable.
\end{thm}

\begin{thm}\label{virt_retract}\cite[Proposition 4.1(1)]{Bestvina}
Let $G$ be a hyperbolic group in the class $\mathcal{Q}\mathcal{V}\mathcal{H}$. Then every quasiconvex subgroup $H$ in $G$ is a virtual retract of $G$.

\end{thm}

\section{Corollaries}

In this section we deduce Corollaries~\ref{Dezember} and~\ref{Dezember1} from Theorem~\ref{hyperb}.
The following lemma is a variant of Lemma 6 from~\cite{ZR}.
We will give its proof for completeness.

\begin{lem}\label{conj_dist} {\rm (see~\cite[Lemma 6]{ZR})} Let $G$ be a conjugacy separable group and let $H$ be a retract of $G$. Then $H$ is conjugacy distinguished.
\end{lem}

{\it Proof.} Let $f:G\rightarrow H$ be a retraction and let $K=\ker(f)$. Then $G=K\rtimes H$.
Let $G_i$,  $i\in I$, be all finite index normal subgroups of $G$.
We set $K_i=G_i\cap K$.
We prove that if $g\in G$ is not conjugate into $H$,
than $g$ is not conjugate into some finite index subgroup of $G$ containing $H$, see Remark~\ref{rmk_CD}.

Suppose the contrary. Then $g$ is conjugate into $K_iH$ for every $i\in I$ (since $K_iH$ has finite index in $G=KH$), i.e., $g^{x_i}=k_ih_i$ for some
$x_i\in G$, $k_i\in K_i$ and $h_i\in H$. This implies $f(g^{x_i})=f(h_i)=h_i$, hence $g^{x_i}=k_if(g)^{f(x_i)}$.
Then $g^{x_i}$ and $f(g)^{f(x_i)}$ have the same images in $G/K_i$, and hence in $G/G_i$.
In particular, $g$ is conjugate to $f(g)$ modulo each $G_i$.
Since $G$ is conjugacy separable, $g$ is conjugate to $f(g)$. Hence, $g$ is conjugate into $H$, a contradiction.
\hfill $\Box$

\medskip

\begin{lem}\label{Dezember}
Let $G$ be a conjugacy separable torsion-free hyperbolic group. If $H_1$ is a retract of~$G$,
then $H_1$ is con-separated from every subgroup of $G$ that is not already conjugate into $H_1$.
\end{lem}

{\it Proof of Lemma~\ref{Dezember} from Corollary~\ref{hyperb_1}.} Let $H_2$ be an arbitrary subgroup of $G$ that is not conjugate into $H_1$.
By Corollary~\ref{hyperb_1}, there is therefore an element
$h\in H_2$ that already is not conjugate into $H_1$. By assumption, $G$ is conjugacy separable. By Lemma~\ref{conj_dist}, the group $H_1$ is conjugacy distinguished in $G$. Thus, there is a finite quotient of $G$
witnessing that $h$ is not conjugate into $H_1$. This quotient then
also witnesses that $H_2$ is not conjugate into $H_1$.\hfill $\Box$

\medskip

The following corollary is the ``heredity'' version of Lemma~\ref{Dezember}.

\begin{cor}\label{Dezember1}
Let $G$ be a torsion-free hyperbolic group with the property~LR.
If $G$ is hereditarily CS, then $G$ is hereditarily SICS and SCS.
\end{cor}

{\it Proof.}
The property~LR is inherited by subgroups.
Therefore and by Corollary~\ref{reduction_free},
it suffices to prove that $G$ is SICS.

We have to show that every finitely generated
subgroup $H_1\leqslant G$ is con-separated in $G$. By assumption,
$H_1$ is a retract of a finite index subgroup $L\leqslant G$.
By Lemma~\ref{overgroups}, it suffices to show that $H_1$ is con-separated
in $L$ from every subgroup of $L$.
Since $L$ is conjugacy separable, we are in the situation of Lemma~\ref{Dezember}
and the claim follows.\hfill $\Box$

\begin{cor}\label{main_cor}
Every torsion-free hyperbolic groups from the class $\mathcal{Q}\mathcal{V}\mathcal{H}$
is hereditarily  quasiconvex-SICS and hereditarily  quasiconvex-SCS. In particular,
the following groups possess these properties:

\begin{enumerate}


\item[(1)] Fundamental groups of closed hyperbolic {\rm 3}-manifolds.

\item[(2)] Torsion-free groups with finite $C'(1/6)$ or $C'(1/4)-T(4)$ presentations.

\item[(3)] Surface groups.
\end{enumerate}





\end{cor}

{\it Proof.}
The first statement follows from Lemma~\ref{Dezember}, since the assumptions of this lemma are fulfilled by Theorems~\ref{herid_CS} and~\ref{virt_retract}.
The groups in (1)-(3) belong to the considered class.
\hfill $\Box$

\section{Hyperbolic groups and SICS}

The main result in this section is Theorem~\ref{hyperb}.
We use the following result of B.H. Neumann.

\begin{thm}\cite[Lemma~4.1]{BN1}\label{Neumann}
Suppose that $(H_i)_{i\in I}$ is a finite family of subgroups of a group $G$ and $(x_i)_{i\in I}$ is a finite family of elements of $G$ with the property $G=\underset{i\in I}{\cup} H_ix_i$. Then there exists $i\in I$
such that $H_i$ has a finite index in $G$.
\end{thm}

The notion of a hyperbolic group derives from the notion of a hyperbolic space. There are several equivalent definitions. We use:

\begin{defn} {\rm (see~\cite{Br})
Let $(\mathcal{X},d)$ be a geodesic metric space.
To shorten notation, we use $|AB|$ for $d(A,B)$.
Any geodesic segment between $A,B$ is denoted by $[A,B]$.
More generally, any geodesic $n$-gon with vertices $A_1,A_2,\dots,A_n$ is denoted by $[A_1,A_2,\dots,A_n]$.

Let $A,B,C\in \mathcal{X}$. The {\it Gromov inner product} of $B,C$ with respect $A$ is
$$(B,C)_A=\frac{1}{2}\bigl(|AB|+|AC|-|BC|\bigr).$$

Let $\delta\geqslant 0$.
A geodesic triangle $[A_1,A_2,A_3]$ in $\mathcal{X}$ is called {\it $\delta$-thin} if for any vertex $A_i$ and any two
points $P\in [A_i,A_j]$, $Q\in [A_i,A_k]$, where $\{i,j,k\}=\{1,2,3\}$ the following implication is satisfied:
 $$
|A_iP|=|A_iQ|\leqslant (A_j,A_k)_{A_i}\hspace*{2mm} \Longrightarrow  \hspace*{2mm} |PQ|\leqslant \delta.$$
A geodesic metric space $\mathcal{X}$ is called {\it $\delta$-hyperbolic} if every geodesic triangle in
$\mathcal{X}$ is $\delta$-thin.
}\end{defn}

As said above, there are several definition. They all agree on what is considered a hyperbolic space. However, they differ with respect to the hyperbolicity constant. A feature common to all definitions of $\delta$-hyperbolicity is that a $\delta$-hyperbolic space is $\delta'$-hyperbolic for each $\delta'\geqslant\delta$. Since there are only finitely many such definitions used in the literature, we choose a $\delta$ such that our space is $\delta$-hyperbolic in any sense used. This trick allows us to quote results from various sources without worrying which definition is employed. The alternative would be to convert hyperbolicity constants, which we consider more confusing.

Directly from this definition it follows that each side of a $\delta$-thin triangle is contained in the
$\delta$-neighborhood of the union of the other two sides. This can be turned into one of the competing definitions.

We use the following two lemmas about the closeness of a polygonal line in the $\delta$-hyperbolic space
to a geodesic segment connecting the endpoints of this line.

\begin{lem}\cite[Lemma 21]{Olsh_1}\label{polygon}
Let $c\geqslant 14\delta$ and $c_1>12(c+\delta)$, and suppose that a geodesic $n$-gon $[A_1,\dots ,A_n]$
in a $\delta$-hyperbolic metric space satisfies the conditions $|A_{i-1}A_i|>c_1$ for $i=2,\dots ,n$
and $(A_{i-2},A_i)_{A_{i-1}}<c$ for $i=3,\dots ,n$. Then the polygonal line
$\rho=[A_1,A_2]\cup [A_2,A_3]\cup \dots \cup [A_{n-1},A_n]$ is contained in the $2c$-neighborhood
of the side $[A_n,A_1]$, and the side $[A_n,A_1]$ is contained in the $14\delta$-neighborhood of $\rho$.
\end{lem}

\begin{lem}\label{close} Let $[A,B,C,D]$ be a geodesic rectangle in a $\delta$-hyperbolic metric space. Then the following statements are satisfied.

\vspace*{1mm}
\begin{enumerate}
\item[{\rm (1)}] The polygonal line $[A,B]\cup [B,C]\cup [C,D]$ is contained in the $2\Delta$-neighborhood of $[A,D]$,
where $\Delta=\max\{|AB|,|CD|\}+\delta$.

\vspace*{1mm}

\item[{\rm (2)}] Suppose such that $|BC|>|AB|+|CD|+2\delta$. Let $[B_1,C_1]$ be the subsegment of $[B,C]$ such that $|BB_1|=|AB|+\delta$ and $|C_1C|=|CD|+\delta$. Then $[B_1,C_1]$ is contained in the $2\delta$-neighborhood of $[A,D]$.
\end{enumerate}
\end{lem}

\medskip

{\it Proof.} First we prove (2); see Fig. 1. Let $X\in [B_1,C_1]$. 
Since $|BX|\geqslant |BB_1|=|BA|+\delta$, there exists $Y\in [A,C]$ with $|XY|\leqslant \delta$. Since $|CY|\geqslant |CX|-|XY|\geqslant |CC_1|-\delta=|CD|$,
there exists $Z\in [A,D]$ with $|YZ|\leqslant \delta$.
Then $|XZ|\leqslant 2\delta$.

To prove the statement (1), we note that

{\small $\bullet$}  $[A,B]\cup [B,B_1]$ is contained in the $2\Delta$-neighborhood of $A$,

{\small $\bullet$} $[D,C]\cup [C,C_1]$ is contained in the $2\Delta$-neighborhood of $D$,

{\small $\bullet$} $[B_1,C_1]$ is contained un the $2\delta$-neighborhood of $[A,D]$ by (2).
\hfill $\Box$

\vspace*{-27mm}
\hspace*{-10mm}
\includegraphics[scale=0.7]{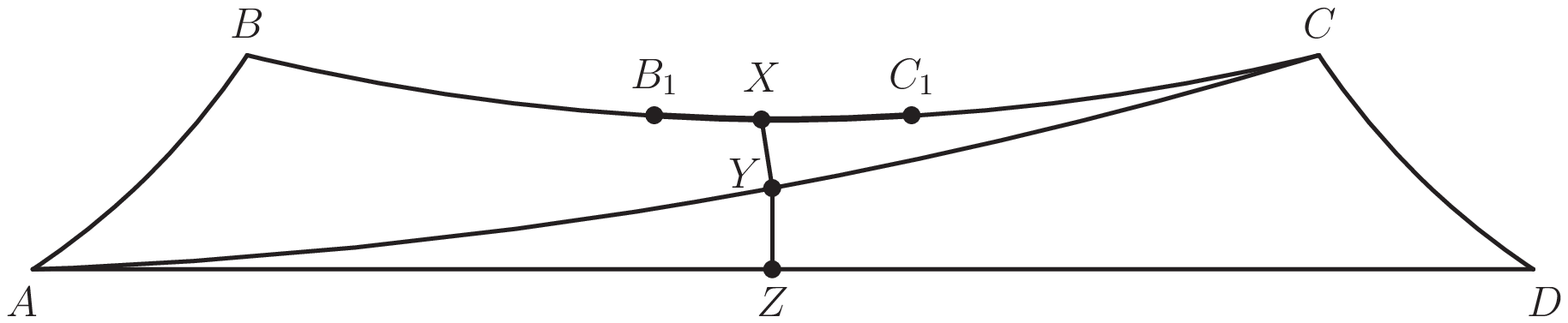}

\vspace*{-14.5cm}

\begin{center}
Figure 1.
\end{center}

We will often use the following lemma.

\begin{lem}\label{Bounded_Hausdorff} (see~\cite[Chapter III.H, Theorem 1.7]{Br}) For all $\delta\geqslant 0$, $\lambda\geqslant 1$, $\epsilon\geqslant 0$ there exists a constant
$R=R(\delta,\lambda,\epsilon)$ with the following property:

If $\mathcal{X}$ is a $\delta$-hyperbolic space, $\rho$ is a $(\lambda,\epsilon)$-quasi-geodesic in $\mathcal{X}$ and $[p,q]$
is a geodesic segment joining the endpoints of $\rho$, then the Hausdorff distance between $[p,q]$ and the image of $\rho$ is less than $R$.
\end{lem}

The {\it length} of an element $g\in G$ with respect to $S$, denoted $|g|$, is the minimal nonnegative integer number $k$ such that $g=x_1\dots x_k$,
where $x_1,\dots,x_k\in S\cup S^{-1}$. The Cayley graph $\Gamma(G,S)$ is endowed by the metric, for which every
edge has length~1.

A finitely generated group $G$ is called {\it$\delta$-hyperbolic} with respect to a finite generating set $S$
if the Cayley graph $\Gamma(G,S)$ is a $\delta$-hyperbolic metric space.
It is well known that changing the finite generating set leads to changing the hyperbolicity constant $\delta$ (see~\cite{Br}). Thus, the notion a hyperbolic group is well defined.


\begin{defn} {\rm Let $G$ be a group with a fixed finite generating set $S$. For elements $u,v\in G$ and a real number $c>0$, we write $uv=u\underset{c}{\cdot}v$ if
$\frac{1}{2}(|u|+|v|-|uv|)<c$.\break
In other words, the Gromov inner product $(u^{-1},v)_1$ is smaller than $c$.
Also we write $uvw=u\underset{c}{\cdot}v\underset{c}{\cdot}w$ if $uv=u\underset{c}{\cdot}v$ and
$vw=v\underset{c}{\cdot}w$.}
\end{defn}

\begin{lem}\cite[Lemma 2.22]{BV}\label{quasi_1} Let $G$ be a $\delta$-hyperbolic group with respect to a finite generating set.
Let $a,b\in G$, $b\neq 1$. For every integer $k\geqslant 0$ and every $z\in G$, there exists $x\in G$
and $0\leqslant l\leqslant k$ such that $z^{-1}ab^kz=x^{-1}\underset{c}{\cdot}b^{k-l}ab^l\underset{c}{\cdot} x$, where $c=c(a,b)$ does not depend on $z$ and $k$.
\end{lem}




\medskip

Recall that a group is called {\it elementary} if it contains a cyclic subgroup of finite index.
For every hyperbolic group $G$ and every element $g\in G$ of infinite order,
there exists a unique maximal elementary subgroup containing $g$, which is denoted by $E(g)$,
see~\cite{Gromov, Olsh}. Note that $E(g)$ coincides with the commensurator of the cyclic group $\langle g\rangle$ that coincides with the stabilizer of the unordered pair of endpoints of the quasi-geodesic $i\rightarrow g^i$, $i\in \mathbb{Z}$.

\begin{lem}\label{quasi_geod_oleg}
Let $G$ be a hyperbolic group with a fixed finite generating set $S$.
Suppose that $b\in G$ is an element of infinite order and $a\in G\setminus E(b)$. Then there exist real numbers $\lambda\geqslant 1$ and $\epsilon\geqslant 0$ such that
for all nonnegative integers $k,l$ the curve $b^kab^l$ is a $(\lambda,\epsilon)$-quasigeodesic.
\end{lem}

{\it Proof.} 
Let $\Gamma$ be the Cayley graph of $G$ with respect to $S$ and let $\partial \Gamma$ be the boundary of $\Gamma$.
We consider a bi-infinite curve $\gamma$ in $\Gamma$ with the label $\dots bbb\cdot a\cdot bbb\dots$.
We write $\gamma=\beta_1\alpha\beta_2$, where the labels of $\beta_1$, $\alpha$, and $\beta_2$ are
$\dots bbb$, $a$, and $bbb\dots$, respectively.
Since $b$ has infinite order, the curves $\beta_1$ and $\beta_2$ are quasi-geodesics with one end on $\partial\Gamma$.

Suppose that $\gamma(-\infty)\neq \gamma(\infty)$. By Lemma 3.2 from~\cite[Chapter III.H]{Br}, there exists a geodesic line $c:\mathbb{R}\rightarrow \Gamma$
with $c(-\infty)=\gamma(-\infty)$ and $c(\infty)=\gamma(\infty)$.
We split $c$ into two rays $c=c_1c_2$.
Since the ray $c_2$ and the quasi-geodesic $\beta_2$ have the same end on the boundary, they are uniformly close
(see Lemma 3.3 from~\cite[Chapter III.H]{Br}). Analogously, $c_1$ and $\beta_1$ are uniformly close.
Therefore, $c$ and $\gamma$ are uniformly close. This implies that $\gamma$ is a quasi-geodesic.

Now suppose that $\gamma(-\infty)=\gamma(\infty)$.
Then the rays $\beta_1^{-1}$ and $\beta_2$ have the same end on the boundary.
Let $A$ and $B$ be their initial points in $\Gamma$.
Since $\beta_1$ and $\beta_2$ are uniformly close, there exists a constant $K>0$ such that
$d(Ab^{-t},Bb^t)\leqslant K$ for all $t\in \mathbb{N}$.
Therefore there exist distinct $t_1,t_2\in \mathbb{N}$ such that $(Ab^{-t_1})^{-1}(Bb^{t_1})=(Ab^{-t_2})^{-1}(Bb^{t_2})$. Note that $a=A^{-1}B$. Setting $s:=t_1-t_2$, we deduce
that $a^{-1}b^{s}a=b^{-s}$.
By Theorem~\cite[Chapter 8, Theorem 37]{Ghys}, a subgroup of a hyperbolic group is ether virtually cyclic,
or contains a nonabelian free group. Therefore $\langle a,b\rangle$ is virtually cyclic, hence $a\in E(x)$,
a contradiction.\hfill $\Box$


\begin{cor}\label{quasi_2} Let $G$ be a $\delta$-hyperbolic group with respect to a finite generating set~$S$.
Suppose that $b\in G$ has infinite order and $a\in G\setminus E(b)$.
Then there exist increasing linear functions $f:\mathbb{R_+}\rightarrow \mathbb{R}$ and $h:\mathbb{R_+}\rightarrow \mathbb{R}_+$ with the following property:

if $\rho$ is the curve in $\Gamma$ with the label
$x^{-1}\underset{c}{\cdot} b^kab^l\underset{c}{\cdot} x$, where $x\in G$, $c\geqslant 14\delta$,
and $k,l\geqslant0$, $k+l\geqslant f(c)$, and $[P,Q]$ is a geodesic segment joining the endpoints of $\rho$,
then $\rho$ is contained in the $h(c)$-neighborhood of $[P,Q]$.
\end{cor}

{\it Proof.}
We set $$f(t):=\frac{12(t+\delta)\lambda-|a|+\lambda\epsilon+1}{|b|},\hspace*{10mm} h(t)=R(\delta,\lambda,\epsilon)+24t+26\delta,$$
where $\lambda,\epsilon$ are from Lemma~\ref{quasi_geod_oleg} and $R(\delta,\lambda,\epsilon)$
is from Lemma~\ref{Bounded_Hausdorff}. Let $k,l\geqslant 0$ and $k+l\geqslant f(c)$.

We write $\rho=\rho_1\rho_2\rho_3$, where the labels of $\rho_1$, $\rho_2$, $\rho_3$ are $x^{-1}$,
$b^{k}ab^l$, $x$, respectively. Consider the polygonal line $\rho'=\rho_1\rho_2'\rho_3$,
where $\rho_2'$ is a geodesic connecting the endpoints of $\rho_2$ (see Fig.~2).
By Lemma~\ref{Bounded_Hausdorff}, $\rho_2$ lies in the $R=R(\delta,\lambda,\epsilon)$-neighborhood of $\rho_2'$. Thus, it suffices to prove that $\rho'$ lies in the $(24c+26\delta)$-neighborhood of $[P,Q]$.

Since $\rho_2$ is a $(\lambda,\epsilon)$-quasi-geodesic, we have
$$l(\rho_2')=|b^{k}ab^l|\geqslant \frac{(k+l)|b|+|a|}{\lambda}-\epsilon\geqslant \frac{(f(c))|b|+|a|}{\lambda}-\epsilon>12(c+\delta).$$

If $|x|>12(c+\delta)$, then $\rho'$ satisfies the assumptions of Lemma~\ref{polygon};
hence $\rho'$ is contained in the $2\delta$-neighborhood of $[P,Q]$.

If $|x|\leqslant 12(c+\delta)$, then by Lemma~\ref{close}(1),
$\rho'$ is contained in the $2\Delta$-neighborhood of $[P,Q]$, where
$2\Delta=2(|x|+\delta)\leqslant 24c+26\delta$.
\hfill $\Box$





\vspace*{-30mm}
\hspace*{0mm}
\includegraphics[scale=0.8]{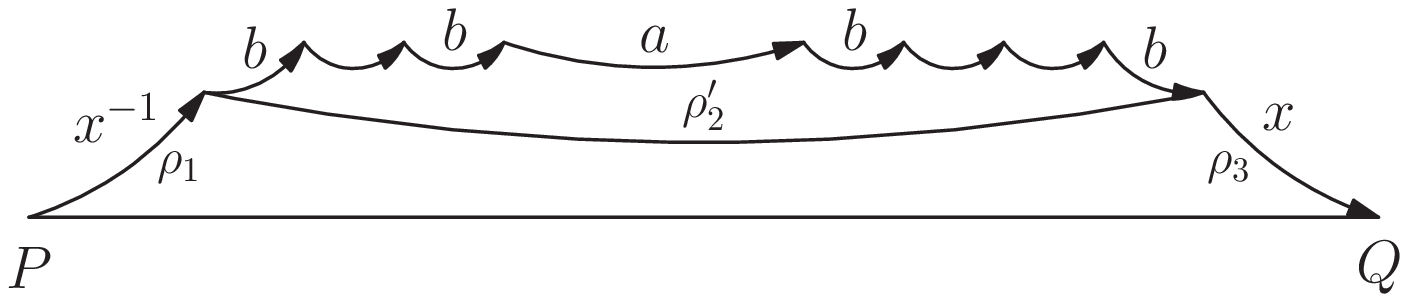}

\vspace*{-17.7cm}

\begin{center}
Figure 2.
\end{center}

\medskip

\begin{lem}\label{quasi_3} Let $G$ be a $\delta$-hyperbolic group with respect to a finite generating set $S$.
Let $H$ be a quasiconvex subgroup of $G$. For every $u,v,b\in G$, where $b$
has infinite order, there exist numbers
$s=s(b)$, $r=r(b)$ (independent of $u,v$), and $n_0=n_0(b,|u|+|v|)$ from $\mathbb{N}$ such that $n_0$ is increasing in the second variable and the following holds:

\medskip

If $ub^nv\in H$ for some $n\geqslant n_0$, then there exists $g\in G$ with $|g|\leqslant r$ such that
\begin{enumerate}
\item[(1)] $g^{-1}b^{s_1}g\in H$ for some $0<s_1\leqslant s$,

\item[(2)] $b^{s_0}v\in gH$ for some $0\leqslant s_0<s_1$.

\end{enumerate}
\end{lem}

\medskip

{\it Proof.}
We choose minimal words representing $u,b,v$ in the alphabet $S$ and denote them by the same symbols.
First we define some constants which we will use later.
\begin{enumerate}
\item[{\small $\bullet$}] Let $[b^i,b^{i+1}]$ be the geodesic segment in $\Gamma(G,S)$ labelled by $b$. The curve $$\underset{i\in \mathbb{Z}}{\cup} [b^i,b^{i+1}]$$ is a $(\lambda,\epsilon)$-quasi-geodesic, where $\lambda,\epsilon$ depend only on $b$.

\medskip




\item[{\small $\bullet$}] Let $R=R(\delta,\lambda,\epsilon)$ be the constant defined in Lemma~\ref{Bounded_Hausdorff} and $\epsilon_1$ be the quasiconvexity constant of $H$. We set $r(b)=R+|b|+2\delta+\epsilon_1$.

\medskip

\item[{\small $\bullet$}] Let $\Delta> 2\delta$ be an arbitrary real number.
We set $n_{\Delta}$ to be the minimal integer such that the following inequality is satisfied:
$$\frac{n_{\Delta}|b|}{\lambda}-\epsilon>3(|u|+|v|+\Delta).$$
\end{enumerate}

Below we will choose an appropriate $\Delta$ and set $n_0:=n_{\Delta}$. Until this moment $\Delta$
is as above.
Let $n\geqslant n_{\Delta}$. We define the following curves in the Cayley graph $\Gamma(G,S)$:

\medskip

--  the geodesic $\alpha$ from $1$ to $u$ with the label $u$,

--  the curve $\beta$ from $u$ to $ub^n$ with the label $b^n$,

--  the geodesic $\gamma$ from $u{b^n}$ to $ub^nv$ with the label $v$.

\medskip

Let $\beta'$ be a geodesic from $u$ to $ub^n$ and $\beta''$ be a geodesic from $1$ to $ub^nv$ (see Fig.~3).
 Since $\beta$ is a $(\lambda,\epsilon)$-quasigeodesic, we have $$l(\beta')\geqslant \frac{l(\beta)}{\lambda}-\epsilon=\frac{n|b|}{\lambda}-\epsilon >
 3(|u|+|v|+\Delta).\eqno{(7.1)}$$

By Lemma~\ref{close}(2), applied to the geodesic rectangle
with the sides $\alpha,\beta',\gamma$, and $\beta''$,
the middle third of $\beta'$ is contained in the $2\delta$-neighborhood of $\beta''$.
Therefore, for each point $B$ in the middle third of $\beta'$, there exists a point $C$ in $\beta''$ with $|BC|\leqslant 2\delta$. Since the geodesic $\beta'$ lies in the $R$-neighborhood of the quasigeodesic $\beta$, there exists $A=ub^t$ in $\beta$ such that $|AB|\leqslant R+|b|$.
Since the geodesic $\beta''$ lies in the $\epsilon_1$-neighborhood of $H$, there exists $D\in H$ such that $|CD|\leqslant \epsilon_1$.
Thus $$|AD|\leqslant R+|b|+2\delta+\epsilon_1=r(b).$$

\vspace*{-22mm}
\hspace*{-10mm}
\includegraphics[scale=0.7]{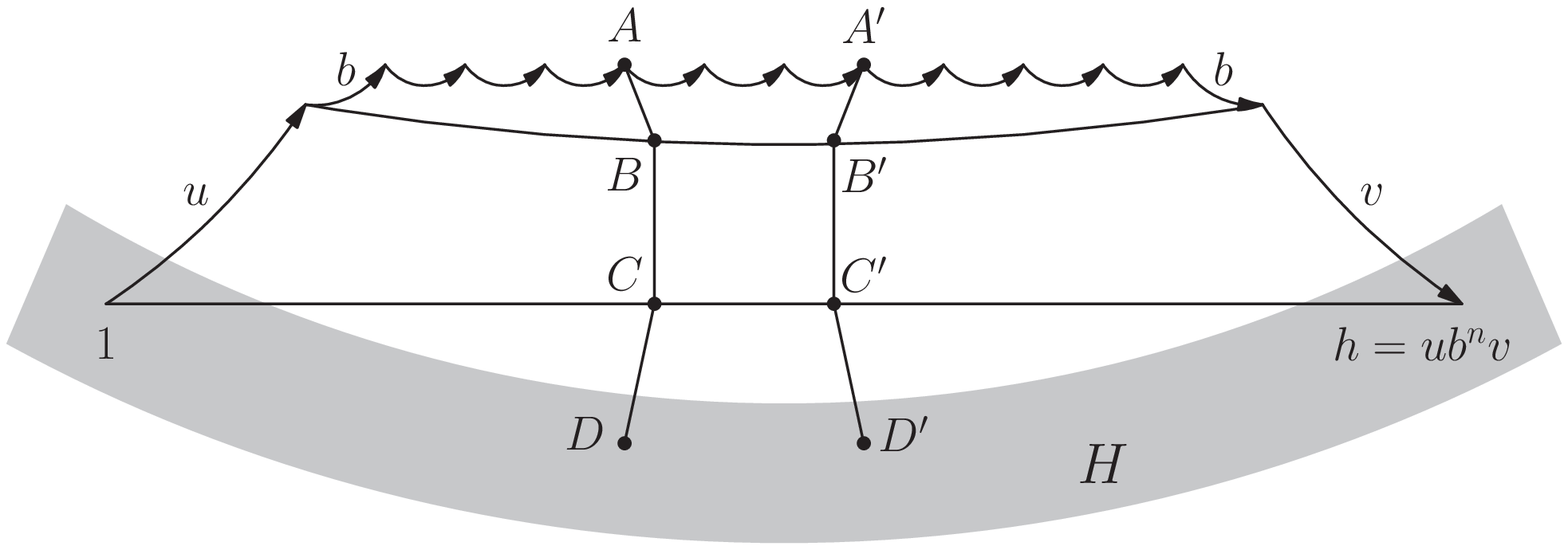}

\vspace*{-12cm}

\begin{center}
Figure 3.
\end{center}

\medskip

By (7.1), the length of the middle third of $\beta'$ is larger than $\Delta$.
If we take $\Delta\geqslant 1+N$, where $N$ is the number of words in the alphabet $S^{\pm}$ with length up to $r(b)$, then  the labels of geodesics $[A,D]$ will repeat when $B$ moves along a segment of length $\Delta$ in the middle third of $\beta'$.
We take $\Delta:=2(R+|b|+1)(2+N)$ to guarantee that there exists a repetition
${\rm label}([A,D])={\rm label}([A',D'])$ with the additional property $$2(R+|b|)<|BB'|< \Delta-2(R+|b|).$$
This implies $0<|AA'|<\Delta$.

Thus, there exist two points $A=ub^t$ and $A'=ub^{t'}$ on $\beta$, and two points $D,D'$ in $H$,
and a constant $s=s(b)$ such that $0<t'-t<s$ and the geodesics $[A,D]$ and $[A',D']$ represent the same element $g\in G$ with $|g|\leqslant r(b)$. The label of the corresponding path $DAA'D'$ is $g^{-1}b^{s_1}g$,
where $s_1=t'-t$. Since $D,D'\in H$, we have property (1).

Then $D=ub^tg\in H$. Write $n-t=ks_1+s_0$ for some integers $k$ and $s_0$ with $0\leqslant s_0<s_1$.
Then $$ub^nv=ub^tg\cdot (g^{-1}b^{s_1}g)^k\cdot g^{-1}b^{s_0}v.$$
Since the elements $ub^nv$, $ub^tg$, and $g^{-1}b^{s_1}g$ lie in $H$, we have $g^{-1}b^{s_0}v\in H$,
and property (2) follows. This shows, that we can take $n_0:=n_{\Delta}$. \hfill $\Box$

\begin{thm}\label{hyperb}
Let $G$ be a hyperbolic group, let $H_1$ be a quasiconvex subgroup of $G$, and let $H_2$
be an arbitrary subgroup of $G$. Suppose that $H_2$ is elementwise conjugate into $H_1$. Then
there exists a finite index subgroup of $H_2$ which is conjugate into~$H_1$.
\end{thm}

\medskip

{\it Proof.} We fix a finite generating set $S$ of $G$. Let $\epsilon_1$ be the quasiconvexity constant for $H_1$.
If all elements of $H_2$ have finite order, then $H_2$ is finite
(see~\cite[Section 8, Theorem 37]{Ghys}), and the statement is trivially valid.
Therefore we suppose that $H_2$ contains an element $b$ of infinite order.
Let $a$ be an arbitrary element in $H_2\setminus E(b)$.
Then for every $k\in \mathbb{N}$, there exists $z_k\in G$ such that $z_k^{-1}ab^kz_k\in H_1$.
By Lemma~\ref{quasi_1}, there exist $x_k\in G$ and $0\leqslant l\leqslant k$ such that
$$x_k^{-1}\underset{c}{\cdot}b^{k-l}ab^l\underset{c}{\cdot} x_k\in H_1$$
for some $c=c(a,b)$. We denote this element by $h_k$ and consider the curve $\gamma_k$
starting at 1 and ending at $h_k$, and having the label $x_k^{-1}b^{k-l}ab^lx_k$ (see  Fig.~3).\break
By Corollary~\ref{quasi_2}, if $k\geqslant f(c)$, then
each point of $\gamma_k$ is at distance at most $h(c)$ from the geodesic $[1,h_k]$,
and hence at distance at most $h(c)+\epsilon_1$ from~$H_1$.



\vspace*{-20mm}
\hspace*{-10mm}
\includegraphics[scale=0.7]{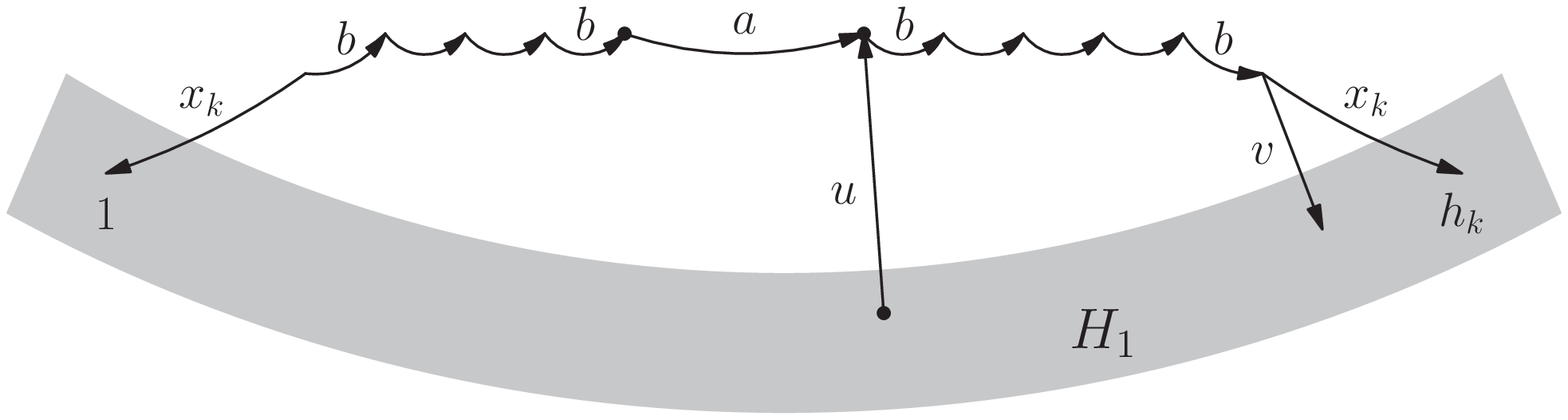}

\vspace*{-13cm}

\begin{center}
Figure 4.
\end{center}

\medskip

We take $k\geqslant f(c)$. Then $k-l$ or $l$ is at least $k/2$. Without loss of generality,
we assume that $l\geqslant k/2$.
Since each point of $\gamma_k$ is at distance at most $h(c)+\epsilon_1$ from $H_1$, there exist $u,v\in G$
such that $$ub^lv\in H_1\hspace*{2mm} {\rm and}\hspace*{2mm} |u|,|v|\leqslant h(c)+\epsilon_1, \hspace*{2mm} {\rm and}\hspace*{2mm} v^{-1}x_k\in H_1$$ (see Fig. 4).
Increasing $k$, we may assume that $l$ is sufficiently large, i.e.,
$$l\geqslant n_0(b,2(h(c)+\varepsilon_1)).$$
By Lemma~\ref{quasi_3}, there exists $g\in G$ with $|g|\leqslant r(b)$ such that

\medskip
\begin{enumerate}
\item[(1)] $g^{-1}b^{s_1}g\in H_1$ for some $0<s_1\leqslant s(b)$,

\item[(2)] $b^{s_0}v\in gH_1$ for some $0\leqslant s_0<s_1$.
\end{enumerate}

\medskip
\noindent
From this and $v^{-1}x_k\in H_1$, we deduce

\begin{enumerate}
\item[(3)] $b^{s_0}x_k=gh_k'$ for some $h_k'\in H_1$.
\end{enumerate}
We have $$H_1\ni x_k^{-1}b^{k-l}ab^lx_k=(h_k'^{-1}g^{-1}b^{s_0})\cdot (b^{k-l}ab^l)\cdot (b^{-s_0}gh_k').$$
Therefore $$g^{-1}b^{k-l+s_0}ab^{l-s_0}g\in H_1.$$
Using (1), we get $g^{-1}b^pab^qg\in H_1$ for some $p,q$ with $0\leqslant p,q\leqslant s(b)$. Thus,
$$a\in b^{-p}gH_1g^{-1}b^p\cdot b^{-(q+p)}.$$
Since $a$ is an arbitrary element in $H_2\setminus E(b)$ and $|g|\leqslant r(b)$, we have
$$H_2\subseteq \underset{(t,z)\in M}{\bigcup}(z^{-1}H_1z\cdot b^{-t})\cup E(b),$$
where $M=\{(t,z)\in \mathbb{Z}\times G\,|\, 0\leqslant t\leqslant 2s(b),\, |z|\leqslant s(b)\cdot |b|+r(b)\}$.
Since $b\in H_2$, we have
$$H_2= \underset{(t,z)\in M}{\bigcup}((z^{-1}H_1z\cap H_2)\cdot b^{-t})\cup (E(b)\cap H_2).$$
Since the set $M$ is finite, we deduce from Theorem~\ref{Neumann} that either $E(b)\cap H_2$ is of finite index in $H_2$, or there exists $(t,z)\in M$ such that
$z^{-1}H_1z\cap H_2$ is of finite index in $H_2$.
In the first case, $\langle b\rangle$ has finite index in $H_2$ and we are done.
In the second case, a finite index subgroup of $H_2$ is conjugate into $H_1$.
\hfill $\Box$




\begin{cor}\label{1} Let $G$ be a hyperbolic group, let $H_1$ be a quasiconvex subgroup
that is closed under taking of roots, and let $H_2$ be an arbitrary subgroup of $G$.
Suppose that $H_2$ is elementwise conjugate into $H_1$. Then $H_2$ is conjugate into~$H_1$.
\end{cor}

{\it Proof.} 
By Theorem~\ref{hyperb}, there is an element $z\in G$ and a subgroup $A$ of a finite index $n$ in $H_2$ such that $z^{-1}Az\leqslant H_1$. Then $z^{-1}H_2^{n!}z\leqslant H_1$. Since $H_1$ is closed under
taking of roots, we have $z^{-1}H_2z\leqslant H_1$.\hfill $\Box$

\begin{cor}\label{hyperb_1}
Let $G$ be a torsion-free hyperbolic group, let $H_1$ be a retract of $G$, and let $H_2$
be an arbitrary subgroup of $G$. Suppose that $H_2$ is elementwise conjugate into $H_1$. Then
$H_2$ is conjugate into $H_1$.
\end{cor}

\medskip

{\it Proof.} 
The group $G$, as every torsion-free hyperbolic group, has the unique root property. By Lemma~\ref{retracts_and_roots},
$H_1$ is closed under taking of roots, and  by Lemma~\ref{observ}, $H_1$ is quasiconvex.
Corollary~\ref{1} completes the proof.\hfill $\Box$



\section{Appendix: Some algorithmic problems related to SCS and SICS}

We say that $G$ has the {\it cyclic-SICS property} if $G$ satisfies the definition of SICS where
$H_1$ runs over all cyclic subgroups of $G$.

In parallel with the properties SCS, SICS, and cyclic-SICS, we discuss the following algorithmic problems,
which are variants of the classical conjugacy problem.

{\bf Subgroup conjugacy problem (P1)}. Given a finitely presented group $G$ and two finitely generated subgroups $H_1$ and $H_2$ of $G$, decide whether there exists $g\in G$ such that $H_1^g=H_2$.

{\bf Subgroup into-conjugacy problem (P2)}.
Given a finitely presented group $G$ and two finitely generated subgroups $H_1$ and $H_2$ of $G$, decide whether there exists $g\in G$ such that $H_1^g\leqslant H_2$.

{\bf Cyclic-subgroup into-conjugacy problem (P3)}. This is the special case of the previous problem where $H_1$
is cyclic.

\medskip

Clearly, the decidability of the word problem for $G$ is a necessary condition for the decidability of each of these
problems for $G$. Therefore all these problems are undecidable in the class of finitely presented groups.




%



\begin{prop}\label{algor}

{\rm (a)} The subgroup conjugacy problem is decidable for SCS groups.

{\rm (b)} The subgroup into-conjugacy problem is decidable for SICS groups.

{\rm (c)} The cyclic-subgroup into-conjugacy problem is  decidable for cyclic-SICS groups.


\end{prop}

{\it Proof.} We prove only (b). The statements (a) and (c) can be proved, analogously.
Let $G$ be a SICS-group given by a finite presentation $\langle  X\,|\  R\rangle$.
We may assume that $G=F(X)/\langle\!\langle R\rangle\!\rangle$, where $F(X)$ is the free group with basis $X$ and $\langle\!\langle R\rangle\!\rangle$ is the normal closure of $R$ in $F(X)$.
Let $H_1$ and $H_2$ be two subgroups of $G$ given by the words $u_1,\dots ,u_n$ and $v_1,\dots ,v_m$
in the alphabet~$X$, respectively.

The algorithm uses two parallel processes.
The first process computes all possible finite images of $G$ and decides, whether the image of
$H_1$ is conjugate into the image of $H_2$. We stop this process with the answer {\bf NO}
if the current image of $H_1$ is not conjugate into the current image of $H_2$. Now, we turn to the second process. Since $G$ is finitely presented, the set
  \[
  \{(u_1^g,\dots ,u_n^g)|\ g\in F(X)\}\bigcap \,
  \underbrace{\widetilde{H}_2\times \dots \times
    \widetilde H_2}_{n\hspace*{2mm} {\rm times}}
  \]
  with
  \(
  \widetilde{H}_2:=\langle v_1,\dots ,v_m\rangle\cdot \langle\!\langle R\rangle\!\rangle
  \)
  is recursively enumerable. So, we start a process enumerating it; but once an element is found, we have a witness that $H_1$ is conjugate into $H_2$. So in this event, we stop with the answer {\bf YES}. Clearly, since $G$ is SICS, one of these processes eventually stops.
\hfill $\Box$

\begin{cor}\label{SCS, but not cyclic-SICS}
There exists a finitely generated torsion free nilpotent group which is SCS, but not cyclic-SICS.
\end{cor}

{\it Proof.}
Theorem C1 in~\cite{Segal2} says that there exist a finitely generated torsion free nilpotent group $G$
and a subgroup $H$ of $G$ such that there is no algorithm to decide, whether a given element $g\in G$ is conjugate
into $H$. Thus, by the statement (c) of Proposition~\ref{algor}, $G$ is not cyclic-SICS.
On the other hand, $G$ is SCS, since by~\cite{GS} all polycyclic-by-finite groups are SCS.\hfill $\Box$

\medskip

The following corollary follows from Corollary~\ref{main_cor}, Proposition~\ref{algor}, and the fact that
every finitely generated subgroup of a surface group is quasi-convex.

\begin{cor}
The algorithmic problems {\bf (P1)} and {\bf (P2)} are decidable for fundamental groups of compact surfaces.
\end{cor}

The algorithm in the proof of Proposition~\ref{algor} does not give an estimate of running time, so it is not useful in practice. If $G$ is a finitely generated free group, then there are polynomial algorithms for {\bf (P1)} and {\bf (P2)}.
We would like to thank Bettina~Eick for questions leading to Proposition~\ref{complexity}.

\begin{prop}\label{complexity}
  There is a polynomial time algorithm to solve problem {\bf (P2)} in a finitely generated free group.
\end{prop}

\begin{proof}
  We regard the free group $F_n=\langle x_1,\ldots,x_n\rangle$ as the fundamental group of a rose $R_n$ with $n$ petals. We orient the edges and label them with letters $x_1,\ldots,x_n$. The subgroup $H_2 \leqslant F_n$ corresponds to a graph $\Gamma_2$ that is a covering space of $R_n$. The subgroup $H_2$ can be identified with the fundamental group of $\Gamma_2$, which involves choosing a vertex $v$ as a base point of $\Gamma_2$. Changing the base vertex corresponds to passing to a conjugate of $H_2$. The edges of $\Gamma_2$ inherit a label and an orientation from the rose $R_n$. The identification of $H_2$ with the fundamental group of $\Gamma_2$ is by reading off the word corresponding to a loop from edge labels and orientations. A reduced word $w\in F_n$ belongs to the subgroup $H_2$ if and only if it describes a loop in $\Gamma_2$ based at $v$.

  Because $H_2$ is finitely generated, $\Gamma_2$ has a finite {\it minimal core}: the minimal core of $\Gamma_2$ is the subgraph consisting of those edges that belong to a cycle in $\Gamma_2$. Equivalently, the minimal core is the smallest deformation retract of $\Gamma_2$. We remark that the core of $\Gamma_2$ can be efficiently constructed from a finite generating set of $H_2$ using Stallings folds. See~\cite{Touikan} for an algorithm that is particularly fast, i.e., almost linear in the total length of the generators. We note that the number of edges (and also the number of vertices) in the core is bounded by the total length of the generators for $H_2$.

  The graph $\Gamma_2$ consists of its core and tree components hanging off the core. From the definition of the minimal core the following is immediate:
\begin{quote}
  Any non-trivial cyclicly reduced closed path in $\Gamma_2$ is contained in its core. Put differently: if reading a cyclicly reduced non-trivial word in $\Gamma_2$ starting at a vertex $v$ takes us back to the vertex $v$, then the vertex $v$ belongs to the core of $\Gamma_2$.
\end{quote}

Let $w_1,w_2,\ldots,w_r$ be reduced words generating $H_1$. By passing to a conjugate of $H_1$, we can assume that at least one generator, say $w_1$ is a non-trivial cyclicly reduced word. This transformation of the problem is computationally cheap.

Now, $H_1$ is conjugate to a subgroup of $H_2$ if and only if there is a vertex $v\in\Gamma_2$ such that all $w_i$ read closed loops based at $v$. By the preceding observation, such a vertex $v$ has to belong to the core of $\Gamma_2$. It follows that testing for such conjugacy amounts to checking each $w_i$ against each vertex $v$ in the core of $\Gamma_2$.

As stated so far, the test requires constructing a path in $\Gamma_2$. However, one can restrict attention to the finite core. A reduced word describes an edge path without backtracking; and if such a path leaves the core (and thus enters a tree component) it can never return to the core. Thus, we can stop checking $w_i$ against $v$ as soon as the constructed path leaves the core of $\Gamma_2$.
\end{proof}
\begin{rmk}{\rm
    Lemma~\ref{observation} shows that a polynomial time algorithm to solve {\bf (P1)} in a finitely generated free group is obtained simply by running the previous algorithm for {\bf (P2)} twice to test whether $H_1$ conjugates into $H_2$ {\it and} $H_2$ conjugates into $H_1$.

}\end{rmk}


\noindent
{\bf Questions.} {\bf (1)}
Does there exist a polynomial algorithm which, given a limit group $G$ and
two finitely generated subgroups $H_1,H_2$ of $G$, decides whether $H_1$ is conjugate into $H_2$?

{\bf (2)} Let $G=G_1\ast_{\mathbb{Z}} G_2$ be a finitely generated amalgamated product over the infinite cyclic group~$\mathbb{Z}$.
Suppose that $G_1$ and $G_2$ are SICS, and that $\mathbb{Z}$ is malnormal in $G$.
Does the group $G$ possess the property SICS?



\begin{rmk}\label{quasi-conv_cannnot be deleted}
{\rm
(1) Recall that an embedding of groups $i:A\hookrightarrow B$ is called {\it malnormal} if
for each $b\in B$ the inequality $i(A)^b\cap i(A)\neq 1$ implies $b\in i(A)$.
We observe that if one of the algorithmic problem {\bf (P1), (P2)}, or {\bf (P3)} is undecidable for $A$ and there exists a malnormal embedding of $A$ into $B$, then this problem is undecidable for $B$ too.


(2) Let $F_s$ be the free group with basis $\{x_1,\dots ,x_s\}$.
We claim that for each $F_n\times F_m$ with $n,m\geqslant 2$ the problem {\bf (P3)} is undecidable.
This can be proved by using Mihailova's construction, see~\cite{Mih} or~\cite{LS}.
Indeed, let $H=\langle x_1,\dots ,x_s\,|\, r_1,\dots ,r_t\rangle$ be a finitely presented group,
and let $L_H$ be the subgroup in $F_s\times F_s$  generated by the pairs
$$
\begin{array}{ll}
(x_i,x_i), & \hspace*{2mm} i=1,\dots, s,\vspace*{1mm}\\
(1,r_j), & \hspace*{2mm} j=1,\dots, t.
\end{array}
$$
By~\cite[Lemma 4.2 in Chapter IV]{LS}, a pair $(u,v)$ from $F_s\times F_s$ lies in $L_H$ if and only if $u=v$ in~$H$.
Hence, the cyclic subgroup $\langle (u,1)\rangle$ is conjugate into $L_H$ if and only if $u=1$ in~$H$.
Therefore, if $H$ is a group with unsolvable word problem, then there is no algorithm to decide
whether a cyclic subgroup of $F_s\times F_s$ is conjugate into $L_H$. Hence {\bf (P3)} is undecidable
for $F_s\times F_s$.
Since for every $n,m\geqslant 2$, there exists a malnormal embedding $i:F_s \times F_s\hookrightarrow  F_n\times F_m$, this problem is also undecidable for $F_n\times F_m$.

\medskip

(3) We claim that for each $F_n\times F_m$ with $n,m\geqslant 2$ the problem {\bf (P1)} is undecidable.
Using malnormal embeddings, it suffices to consider the case $F_6\times F_6$.
For this group the generation problem is undecidable, see~\cite{Miller} or~\cite[Theorem 4.4 in Chapter IV]{LS}. This means that it is undecidable whether a finite list of elements generates $F_6\times F_6$. The following observation completes the proof: if $G$ is a group for which the problem {\bf (P1)} is decidable, then the generation problem for $G$ is also decidable.

\medskip

(4) From (2) and (3), and using Proposition~\ref{algor}, we get that the groups $F_n\times F_m$ with $n,m\geqslant 2$ are not cyclic-SICS and not SCS.
(An alternative proof that $F_2\times F_2$ is not SCS can be found in~\cite{CZ3}.)
The latter stands in contrast to the fact that these groups are conjugacy separable.
We do not know, whether $\mathbb{Z}\times F_2$ SCS or not.
}
\end{rmk}




\begin{prop}\label{CS but not SICS/SCS}
{\it There exists a torsion-free group with finite $C'(1/6)$ presentation that
is hereditarily conjugacy separable, but is not SICS and is not SCS.
}
\end{prop}


{\it Proof.}
We will use the Rips construction~\cite{Rips} and the following easy observation. Suppose that $N$ is a finitely generated normal subgroup of a group $G$. If the problem {\bf (Pi)} is undecidable for $G/N$, then it is undecidable for $G$.

Given any finitely presented group $Q$, Rips constructs a group $G$
with a finite $C'(1/6)$-presentation and a finitely generated normal subgroup $N$ of $G$
such that $G/N\cong Q$.
We take $Q=F_2\times F_2$ in this construction.
By Remark~\ref{quasi-conv_cannnot be deleted}(2)-(3), the problems {\bf (P2)} and {\bf (P1)} are undecidable for $Q$,
hence they are undecidable for $G$.
By Proposition~\ref{algor}, $G$ is not SICS and is not SCS.

The group $G$ is hereditarily conjugacy separable by Theorem~\ref{herid_CS} combined with
Lemma~\ref{examples_1}(4). In the construction of Rips, one can additionally provide that every defining relation of $G$ is not a proper power.
Then $G$ is torsion-free by the result of Greendlinger~\cite{Greendlinger} saying
that any group with a finite $C'(1/6)$-presentation, where each relator is not a proper power,
is torsion-free. \hfill$\Box$


\begin{thebibliography}{AA}

\bibitem{Agol} I. Agol, {\it The virtual Haken conjecture}, Doc. Math., {\bf 18} (2013), 1045--1087.
With an appendix by Ian Agol, Daniel Groves, and Jason Manning.



\bibitem{BogBux} O. Bogopolski, K.-U.Bux, {\it Subgroup conjugacy separability for surface groups}.
Preprint (2014), 22 pages.
Available at http://arxiv.org/pdf/1401.6203.pdf

\bibitem{BV} Bogopolski O., Ventura E., {\it On endomorphisms of torsion-free hyperbolic groups}, International Journal of Algebra and Computations, {\bf 21}, (8) (2011), 1415-1446.


\bibitem{BW} N. Bergeron, D.T. Wise, {\it A boundary criterion for cubulation}, Amer. J. Math. {\bf 134} (3) (2012),
843-859.

\bibitem{Bestvina} M. Bestvina, {\it Geometric group theory and 3-manifolds hand in hand: The fulfillment of Thurston's vision}, Bull of the Amer. Math. Soc., {\bf 51} (1) (2014), 53-70.

\bibitem{BG} O. Bogopolski, F.J. Grunewald, {\it On subgroup conjugacy separabiltiy in the class of virtually free groups}, Max-Planck-Institute of Mathematics Preprint Series, n.~110 (2010), 18~pages.

\bibitem{Br} M.R. Bridson, A. Haeffliger, {\it Metric spaces of non-positive curvature}, Springer, 1999.



\bibitem{CZ2} S.C. Chagas, P.A. Zalesskii, {\it Limit groups are conjugacy separable}, Internat. J. Algebra Comput
{\bf 17} (4) (2007), 851-857.

\bibitem{CZ3} S.C. Chagas, P.A. Zalesskii, {\it Limit groups are subgroup conjugacy separable}, Preprint, 2015.
http://arxiv.org/pdf/1511.04607.pdf


\bibitem{Friedl_2} S. Friedl, {\it An introduction to 3-manifolds and their fundamental groups},
Preprint.

\bibitem{Ghys} E. Ghys and P. de la Harpe, {\it Sur le  Groupes Hyperboliques d'apr$\grave{e}$s Michael Gromov}, Progress in Mathematics, {\bf 83},  Birkhauser, 1990.

\bibitem{Gitik_1} R. Gitik, {\it On quasiconvex subgroups of negatively curved groups},
J. of Pure and Appl. Algebra, {\bf 119} (2) (1997), 155-169.


\bibitem{GMRS} R. Gitik, M. Mitra, E. Rips, M. Sageev, {\it Widths of subgroups}, Trans. Amer. Math. Soc.,
{\bf 350} (1998), 321--329.

\bibitem{Greendlinger} M.D. Greendlinger,
{\it Dehn's algorithm for the word problem}, Comm Pure Appl. Math.,
{\bf 13}, 67-83.

\bibitem{Gromov} M. Gromov, {\it Hyperbolic groups}, Essays in group theory, ed. S.M.~Gersten, Math. Sci. Res. Inst. Publ., Vol. {\bf 8}, Springer, 1987, pp. 75-263.


\bibitem{GS} F.J. Grunewald and D. Segal, {\it Conjugacy in polycyclic groups}, Comm. Algebra {\bf 6} (1978), 775-798.

\bibitem{HW1} F. Haglund, D.T. Wise, {\it Coxeter groups are virtually special}, Advances in Mathematics, {\bf 224} (2010), 1890-1903.

\bibitem{HW2} F. Haglund, D.T. Wise, {\it Special cube complexes}, Geom. Funct. Anal., {\bf 17} (5) (2008),
1551-1620.

\bibitem{Hall} M. Hall, {\it Coset representations in free groups} Trans. Amer. Math. Soc., {\bf 67} (1949), 421--432.

\bibitem{HWZ} E. Hamilton, H. Wilton, P.A. Zalesskii, {\it Separability of double cosets and conjugacy classes in
3-manifold groups}, J. London Math. Soc. {\it 2} {\bf 87} (2013), no. 1, 269--288.

\bibitem{HW} C.C. Hruska, D.T. Wise, {\it Towers, ladders and B.B. Neumann spelling theorem},
 J. Austral. Math. Soc. {\bf 71} (no.1) (2001), 53–-69.

\bibitem{KM} O. Kharlampovich, A. Myasnikov, {\it Hyperbolic groups and free constructions}, Trans. Amer. Math. Soc.,
{\bf 350} (2) (1998), 571-613.

\bibitem{LR} D.D. Long, A.W. Reid, {\it Subgroup separability and local retractions of groups}, Topology {\bf 47} (2008), 137--159.

\bibitem{LS} R. Lyndon and P. Schupp, Combinatorial group theory, Springer-Verlag, 1977.


\bibitem{Mih} K. A. Mihailova, {\it The occurrence problem for direct products of groups}, Dokl. Acad. Nauk SSRR
\textbf{119} (1958), 1103-1105.

\bibitem{MT} M.L. Mihalik, W. Towle, {\it Quasiconvex subgroups of negatively curved groups}, J. Pure Appl. Algebra,
{\bf 95} (1994), 297-301.

\bibitem{Miller} Miller C.F. III, {\it On group-theoretic decision problems and their classification},
Ann. Math. Studies, {\bf 68} (1971), Princeton Univ. Press.

\bibitem{M1} A.Minasyan, {\it Hereditary conjugacy separability of right angled Artin groups and its applications},
Groups, Geometry and Dynamics, {\bf 6} (2) (2012), 335-388

\bibitem{MZ} A. Minasyan, P. Zalesskii, {\it One-relator groups with torsion are conjugacy separable},
Journal of Algebra, {\bf 382} (2013), 39--45.

\bibitem{MZ1} A. Minasyan, P. Zalesskii, {\it Virtually compact hyperbolic groups are conjugacy separable}, (2015)
http://arxiv.org/abs/1504.00613



\bibitem{BN1} B.H. Neumann, {\it Groups covered by permutable subsets.} J. London Math Soc., {\bf 29} (1954), 236-248.

\bibitem{NR} G.A. Niblo, L.D. Reeves, {\it Coxeter groups act on CAT(0) cube complexes},
J. Group Theory, {\bf 6} (3) (2003), 399-413.


\bibitem{Olsh} A. Yu. Ol'shanskii, {\it On residualing homomorphisms and $C$-subgroups of hyperbolic groups},
Inern. J. Algebra and Comput., {\bf 3} (4) (1993), 365-409.

\bibitem{Olsh_1} A. Yu. Ol'shanskii, {\it Periodic quotients of hyperbolic groups}, Mat. Zb., {\bf 184} (4) (1991), 543-567 (in Russian), Engl. translation in {\sl Math. USSR Sb.} {\bf 72} (2) (1992).


\bibitem{ZR} L. Ribes, P. Zalesskii, {\it Conjugacy distinguished subgroups}, Preprint (2015). Available from http://arxiv.org/pdf/1504.02982v2.pdf

\bibitem{Rips} E. Rips, {\it Subgroups of small cancellation groups}, Bull. London Math. Soc., {\bf 14} (1982), 45-47.

\bibitem{Sageev} M. Sageev, {\it Ends of groups pairs and non-positively curved cube complexes}, Proc. London Math. Soc., {\bf 71} (1995), 585-617.

\bibitem{Scott} P. Scott, {\it Subgroups of surface groups are almost geometric}, J. London Math. Soc.,
{\bf 17} (3) (1978), 555-565. See also ibid Correction: J. London Math. Soc., {\bf 32} (2) (1985), 217--220.

\bibitem{Segal} D. Segal, {\it Polycyclic groups}, Cambridge University Press, 1983.

\bibitem{Segal2} D. Segal, {\it Decidable properties of polycyclic groups}, Proc. London Math. Soc., {\bf 61} (3) (1990), 497-528.


\bibitem{Touikan} N. Touikan, {\it A fast algorithm for Stallings' folding process},
Intern. J. of Algebra and Computation, {\bf 16} (6) (2006), 1031-1045.

\bibitem{W} H. Wilton, {\it Hall's theorem for limit groups}, GAFA, v. 18 (2008), 271-303.

\bibitem{Wise1} D. T. Wise, {\it Cubulating small cancellation groups}, Geom. Func.Anal., {\bf 14} (1)
(2004), 150--214.

\bibitem{Wise2} D. T. Wise, {\it The structure of groups with a quasiconvex hierarchy.} Preprint (2011).
Available from http://www. math. mcgill.ca/wise/papers.html.


\end{thebibliography}
\end{document}